\documentclass{amsart}
\usepackage{amsmath, amssymb, mathrsfs, pstricks, pst-node, amsthm, enumerate}
\usepackage{graphicx, rotating,latexsym}
\theoremstyle{definition}
\newtheorem{thm}{Theorem}
\newtheorem{thm-defn}[thm]{Theorem-definition}
\newtheorem{lemma}[thm]{Lemma}
\newtheorem{proposition}[thm]{Proposition}
\numberwithin{thm}{section} \theoremstyle{definition}
\newtheorem{defn}[thm]{Definition}
\newtheorem{corollary}[thm]{Corollary}
\newcommand\Hecke{\mathscr{H}}

\usepackage{epsfig}
\usepackage{graphics}
\begin{document}
\title{$W\!$-graphs for Hecke algebras with unequal parameters (II)}
\author{Yunchuan Yin}
\address{Department of Mathematics \\
Shanghai University of Finance and Economics\\
P.R.China}  


\begin{abstract}
This paper is the continuation of the work in~\cite{Yin}. In that paper we generalized the definition
of $W$-graph ideal in the weighted Coxeter groups, and showed how to construct a $W$-graph from
a given $W$-graph ideal in the case of unequal parameters.

In this paper we study the full $W$-graphs for a given $W$-graph ideal. We show that there exist a pair of dual modules associated with
a given $W$-graph ideal, they are connected by a duality map. 
For each of the dual modules, the associated full $W$-graphs can be constructed.
Our construction closely parallels that of Kazhdan and Lusztig~\cite{KL, Lusztig1, Lusztig2}, which can be regarded as the special case $J=\emptyset$. It also
generalizes the work of Couillens~\cite{C},  Deodhar~\cite{Deodhar, Deodhar2}, and Douglass \cite{MD}, corresponding to
the parabolic cases.
\end{abstract}
\maketitle
\section*{Introduction}
Let $(W, S)$ be a Coxeter system and $\Hecke(W)$ its Hecke algebra over $\mathbb{Z}[q, q^{-1}]$, the ring of Laurent polynomials in the indeterminate $q$. This is now called \emph{the one
parameter case (or the equal parameter case)}. In~\cite{Howlett} Howlett and Nguyen introduced the
concept of \emph{a $W$-graph ideal} in $(W,\leqslant_L)$ with respect to a subset $J$ of $S$, where $\leqslant_L$ is the left weak Bruhat order on $W$. They showed that a $W$-graph can be constructed from a given $W$-ideal, and a
Kazhdan-Lusztig like algorithm was obtained.

In particular, $W$ itself is a $W$-graph ideal with respect to $\emptyset$, and the $W$-graph obtained is the Kazhdan-Lusztig $W$-graph for the regular representation of
$\Hecke(W)$ (as defined in ~\cite{KL}). More generally, it was shown that if $J$ is an arbitrary
subset of $S$ then $D_J$, the set of distinguished left coset representatives of $W_J$ in
$W$, is a $W$-graph ideal with respect to $J$ and also with respect to $\emptyset$, and Deodhar's
parabolic analogues of the Kazhdan-Lusztig construction are recovered.

In \cite{Yin} we generalized the definition
of $W$-graph ideal in the Coxeter groups with a weight function $L$, we showed that the $W$-graph can also be constructed from a given $W$-graph ideal.

In this
paper we continue the work in \cite{Yin}, it grows out of our attempt to understand the "full $W$-graphs" that include $W$-graphs and their dual ones.
The duality has appeared in some literatures, for instance, in the original paper~\cite{KL} Kazhdan and Lusztig
implicitly provided a pair of dual bases $C$ and $C'$ for the Hecke algebras,
Deodhar introduced a pair of dual modules $M^J$ and $\widetilde{M}^J$ in parabolic cases( see \cite{Deodhar, Deodhar2}).

The paper is organised as follows. In Section 1 we present some basic concepts and facts concerning the weighted Coxeter groups,
Hecke algebras and
$W$-graphs. In Section 2, we recall the concept of $W$-graph ideal. 
In Section 3, we show that there exist a pair of dual modules $M(\textbf{E}_J, L)$ and $\widetilde{M}(\textbf{E}_J, L)$
that are associated with a given $W$-graph ideal $\textbf{E}_J$, they are connected
by a duality map, this in turn can be used for the construction of the dual bases of the $W$-graphs. 
This construction closely parallels the work of Deodhar \cite{Deodhar,Deodhar2}, Douglass~\cite{MD},
where they focused primarily upon the parabolic cases.

 In Section 4 we prove in general the construction of another pair of dual $W$-graph bases. This part is motivated by Lusztig's work \cite[Ch. 10]{Lusztig2},
 the construction is obtained by using the bases of $\Hecke$-modules $Hom_A(M, A)$ and $Hom_A(\widetilde{M}, A)$.

In Section 5,  in the case $W$ is finite we prove an inversion formula that relates the two versions of the relative Kazhdan-Lusztig polynomials,
.
  In the last section we give some examples and remarks.
\section{preliminaries}\label{Section1}
Let $W$ be a Coxeter group, with generating set $S$. In this section, we briefly recall some basic concepts concerning the general multi-parameter framework of Lusztig~\cite{Lusztig1, Lusztig2}, which introduces a weight function into Coxeter groups and their associated Hecke algebras on which
all the subsequent constructions depend.

We denote by $\ell: W\rightarrow \mathbb{N}=\{ 0, 1,2,\cdots \}$ the length function on $W$ with respect to $S$. Let $\leqslant$ denote the Bruhat order on $W$.

Let $\Gamma$ be the totally ordered abelian group which will be denoted additively, the order on $\Gamma$ will be denoted by $\leqslant$. Let $\{L(s)\mid s\in S\}\subseteq \Gamma$ be a collection of elements such that $L(s)=L(t)$ whenever $s, t\in S$ are conjugate in $W$. This gives rise to a weight function
\[
L: W\longrightarrow \Gamma
\]
in the sense of Lusztig~\cite{Lusztig1, Lusztig2}; we have $L(w)=L(s_1)+L(s_2)+\cdots +L(s_k)$ where
$w=s_1s_2\cdots s_k (s_i \in S)$ is a reduced expression for $w\in W$. We assume throughout that
\[
L(s)\geqslant 0
\]
for all $s\in S$. (If $\Gamma= \mathbb{Z}$ and $L(s)=1$ for all $s\in S$, then this is the original "equal parameter" setting of ~\cite{KL}).

Let $R\subseteq\mathbb{C}$ be a subring and $A=R[\Gamma]$ be a free $R$-module with basis $\{q^{\gamma}\mid \gamma\in \Gamma\}$ where $q$ is an indeterminant.(The basic constructions in this section are independent
of the choice of $R$ and so we could just take $R = \mathbb{Z})$.  The flexibility of $R$ will be useful once we consider representations of $W$). There is a well-defined ring structure on $A$ such that $q^{\gamma}q^{\gamma'}=q^{\gamma+\gamma'}$ for all $\gamma, \gamma'\in \Gamma$. We denote $1=q^0\in A$. If $a\in A$ we denote by $a_\gamma$ the coefficient of $a$ on $q^{\gamma}$ so that $a=\sum_{\gamma \in \Gamma}a_{\gamma} q^{\gamma}$. If $a\neq 0$ we define the degree of $a$ as the element of $\Gamma$ equal to

\[
\deg(a)=max \{\gamma \mid a_\gamma \neq 0\}
\]
by convention(see~\cite{Bonnafe2}), we set $\deg 0= -\infty$. So $\deg: A\rightarrow \Gamma \cup \{-\infty\}$ satisfies $\deg(ab)=\deg(a)+\deg(b)$.

Let $\Hecke=\Hecke(W, S, L)$ be the generic Hecke algebra corresponding to $(W, S)$ with parameters $\{q^{L(s)}\mid s\in S\}$.  Thus $\Hecke$ has an $A$-basis $\{T_w\mid w\in W  \}$ and the multiplication is given by the rules
\begin{equation}\label{eq:1}
T_sT_w=\begin{cases}
T_{sw}&\text{if $\ell(sw)>\ell(w)$}\\[3 pt]
T_{sw}+(q^{L(s)}-q^{-L(s)} )T_w&\text{if $\ell(sw)<\ell(w)$,}
\end{cases}
\end{equation}
Let $\Gamma_{\geqslant \gamma_0}=\{\gamma\in \Gamma\mid \gamma\geqslant \gamma_0 \}$ and denote by $A_{\geqslant \gamma_0}$ ( or $R[\Gamma_{\geqslant \gamma_0}]$ ) the set of all $R$-linear combinations of terms $q^\gamma$ where $\gamma\geqslant \gamma_0$. The notations $A_{\gamma > \gamma_0}, A_{\gamma \leqslant \gamma_0}, A_{\gamma < \gamma_0}$ have a similar meaning.

 We denote by $A\mapsto A, a\mapsto \overline{a}$ the automorphism of $A$ induced by the automorphism of $\Gamma$ sending $\gamma$ to $-\gamma$ for any $\gamma\in \Gamma$. This extends to a ring involution $\Hecke \mapsto \Hecke, h\mapsto \overline{h}$, where
 \begin{equation*}
 \overline{\sum_{w\in W}a_w T_w}=\sum_{w\in W}\overline{a_w} T_{w^{-1}}^{-1},  \!\!\!\text{ $a_w\in A$ for all $w\in W$},
 \end{equation*}
 and
 \[
 \overline{T_s}=T_s^{-1}=T_s+(q^{-L(s)}-q^{L(s)}) \text{ for all $s\in S$}.
 \]
 \subsection*{ Definition of $W$-graph}
\begin{defn}
 (for equal parameter case see~\cite{KL} ; for general $L$ see~\cite{Geck2} ). A $W$-
graph for $\Hecke$ consists of the following data:
\begin{itemize}
\item[(a)]
 a base set $\Lambda$ together with a map
$I$ which assigns to each $x \in \Lambda$ a subset $I(x) \subseteq S$;
\item[(b)]
 for each $s \in  S$ with $L(s) > 0$, a collection of elements
\begin{equation*}
\{ \mu^s_{x, y} \mid  \text{$ x,y \in \Lambda$  such that $s\in  I(x), s\notin I(y)$} \};
\end{equation*}
\item[(c)]
 for each $s \in S$ with $L(s) = 0$ a bijection $\Lambda\rightarrow \Lambda, x \rightarrow s.x$.
These data are subject to the following requirements.
First we require that, for any $x,  y \in \Lambda$ and $s \in S$ where $\mu^s_{x,y}$
 is defined, we have
 \[
\text{$q^{L(s)}\mu^s_{x, y}\in R[\Gamma_{>0}]$} \text{ and $\overline{\mu^s_{x, y}}=\mu^s_{x, y}$}.
\]
\end{itemize}
Furthermore, let $[\Lambda]_A$ be a free $A$-module with basis $\{b_y\mid y\in \Lambda  \}$. For $s \in S$, define
an $A$-linear map
\begin{equation}
\rho_s(b_y)=\begin{cases}
b_{s.y}&\text{if $ L(s) = 0$};\\[3 pt]
-q^{-L(y)}b_y&\text{if $L(s) > 0, s \in I(y)$;}\\[3 pt]
q^{L(y)}b_y + \sum_{x\in \Lambda; s\in I(x)}\mu_{x, y}^s b_x  &\text{if $L(s) > 0, s \notin I(y)$.}
\end{cases}
\end{equation}
Then we require that the assignment $T_s \mapsto  \rho_s$ defines a representation of $\Hecke$ .
\end{defn}
\section{$W$-graph ideals}
For each $J\subseteq S$, let $\hat{J}=S\backslash J$(the
complement of $J$) and define $W_J=\langle J\rangle$, the
corresponding parabolic subgroup of~$W$.  Let $\Hecke_J$ be the
Hecke algebra associated with~$W_J$. As is well known, $\Hecke_J$
can be identified with a subalgebra of~$\Hecke$.

Let $D_J=\{\,w\in W\mid \ell(ws)>\ell(w)\text{ for all $s\in
J$}\,\}$, the set of minimal coset representatives of~$W/W_J$. The
following lemma is well known.

\begin{lemma}\cite[Lemma 2.1(iii)]{Deodhar}(modified)   \label{lemma:1}Let $J\subseteq S$ and
$s\in S$, and define
\begin{align*}
D_{J,s}^-&=\{\,w\in D_J\mid \text{$\ell(sw)<\ell(w)$}\,\},\\
D_{J,s}^+&=\{\,w\in D_J\mid \text{$\ell(sw)>\ell(w)$ and
   $sw\in D_J$}\,\},\\
D_{J,s}^0&=\{\,w\in D_J\mid \text{$\ell(sw)>\ell(w)$ and
   $sw\notin D_J$}\,\},
\end{align*}
so that $D_J$ is the disjoint union $D_{J,s}^-\cup D_{J,s}^+\cup
D_{J,s}^0$. Then $sD_{J,s}^+=D_{J,s}^-$, and if\/ $w\in D_{J,s}^0$
then $sw=wt$ for some $t\in J$.
\end{lemma}
In this section we shall recall \cite[Section 5]{Howlett}, with some modification.

Let $\leqslant_L$ denote the left weak (Bruhat)order on $W$. We say $x\leqslant_L y$ if and only if
$y=zx$ for some $z\in W$ such that $\ell(y)=\ell(z)+\ell(x)$. We also say that $x$ is a \emph{suffix} of $y$.
The following property of the Bruhat order is useful (see \cite[Corollary 2.5]{Lusztig2}, for example).
\begin{lemma}
Let $y, z\in W$ and let $s\in S$.
\begin{itemize}
\item[(i)]
 Assume that $sz<z$, then $y\leqslant z\Longleftrightarrow sy\leqslant z$.
\item[(ii)]
 Assume that $y<sy$, then $y\leqslant z\Longleftrightarrow y\leqslant sz$.
\end{itemize}
\end{lemma}
\begin{defn}
If $X \subseteq W$, let $Pos(X) = \{ s \in S \mid \ell(xs) > \ell(x) \text{for all} {x \in X} \}$.
\end{defn}
Thus $Pos(X)$ is the largest subset $J$ of $S$ such that $X \subseteq D_J$ .
Let $\textbf{E}$
be an ideal in the poset $(W,\leqslant_L)$; that is, $\textbf{E}$ is a subset of $W$ such that every
$u \in W$ that is a suffix of an element of $\textbf{E}$  is itself in $\textbf{E}$ . This condition implies
that $Pos( \textbf{E} ) = S \backslash \textbf{E} = \{ s \in S \mid s \notin E \}$ . Let $J$ be a subset of
$Pos(\textbf{E})$, so that $\textbf{E} \subseteq D_J$. In contexts we shall denote by $\textbf{E}_J$ for the set $\textbf{E}$, with reference to $J$,   For each $s\in S $ we classify the elements in $\textbf{E}_J$ as follows:
\begin{align*}
\textbf{E}_{J,s}^-&=\{\,w\in \textbf{E}_J\mid \text{$\ell(sw)<\ell(w)$ and $sw\in \textbf{E}_J$}\,\},\\
\textbf{E}_{J,s}^+&=\{\,w\in \textbf{E}_J\mid \text{$\ell(sw)>\ell(w)$ and
   $sw\in \textbf{E}_J$}\,\},\\
 \textbf{E}_{J,s}^{0,-}&=\{\,w\in \textbf{E}_J\mid \text{$\ell(sw)>\ell(w)$ and
   $sw\notin D_J$}\,\},\\
 \textbf{E}_{J,s}^{0,+}&=\{\,w\in \textbf{E}_J\mid \text{$\ell(sw)>\ell(w)$ and
   $sw\in {D_J\backslash \textbf{E}_J} $}\,\}.
\end{align*}

Since $\textbf{E}_J\subseteq D_J$ it is clear that, for each $w \in \textbf{E}_J$, each $s \in S$ appears in exactly one of
the following four sets $SA(w)=\{s\in S\mid w\in  \textbf{E}_{J,s}^+ \}, SD(w)=\{s\in S\mid w\in  \textbf{E}_{J,s}^- \}$,
${WA}_J =\{s\in S\mid w\in  \textbf{E}_{J,s}^{0,+} \} $ and ${WD}_J =\{s\in S\mid w\in  \textbf{E}_{J,s}^{0,-} \} $.
We call the elements of these sets the strong ascents,
strong descents, weak ascents and weak descents of $w$ relative to $\textbf{E}_J$ and $J$.
In contexts where the ideal $\textbf{E}_J$ and the set $J$ is fixed we frequently omit reference to $J$, writing $WA(w)$
and ${WD}(w)$ rather than ${WA}_J (w)$ and ${WD}_J (w)$. We also define the sets of descents
and ascents of $w$ by $D (w) = SD(w) \cup {WD}(w)$ and $A (w) = SA(w) \cup {WA}(w)$.

{\textbf{Remark}}. It follows from ~\ref{lemma:1} that
\begin{align*}
{WA}_J (w) = \{ s \in S \mid \text{$sw \notin \textbf{E}_J$  and $w^{-1}sw \notin J$ }\},\\
{WD}_J (w) = \{ s \in  S \mid \text{$sw \notin \textbf{E}_J$  and $w^{-1}sw \in J$ }\}.
\end{align*}
since $sw \notin \textbf{E}_J$ implies that $sw > w$ (given that $\textbf{E}_J$ is an ideal in $(W,\leqslant L)$). Note also
that $J = {WD}_J (1)$.
\begin{defn}\cite[Definition 5.1]{Howlett}(modified)\label{def3}
 Let $(W,S)$ be a Coxeter group with weight function $L$ such that $L(s)\geqslant 0$ for all $s\in S$, $\Hecke$ be the corresponding Hecke algebra.
 The set $\textbf{E}_J$ is said to be a $W$-graph ideal
with respect to $J (\subseteq S)$ and $L$ if the following hypotheses are satisfied.

(i) There exists an $A$-free $\Hecke$-module $M(\textbf{E}_J, L)$ possessing an $A$-basis

 \[
 B =\{ \Gamma_w | w \in \textbf{E}_J \},
 \]
 for any $s\in S$ and any $w\in \textbf{E}_J $ we have
\begin{equation}\label{eq:4}
T_s\Gamma_w=
\begin{cases}
\Gamma_{sw}+(q^{L(s)}-q^{-L(s)})\Gamma_w &\text{if $w\in \textbf{E}_{J,s}^-$},\\
\Gamma_{sw}
&\text{if $w\in \textbf{E}_{J,s}^+$},\\
-q^{-L(s)}\Gamma_w & \text{if $w\in \textbf{E}_{J,s}^{0,-}$ },\\
q^{L(s)} \Gamma_w-\sum_{\substack{z \in \textbf{E}_J\\z<w}}r^s_{z,w}\Gamma_z& \text{if $w\in \textbf{E}_{J,s}^{0,+}$},
\end{cases}
\end{equation}
for some polynomials $r^s_{z,w}\in q^{L(s)}A_{>0}$.
(ii) The module $M(\textbf{E}_J, L)$ admits an $A$-semilinear involution $\alpha\mapsto \overline{\alpha}$ satisfying $\overline{\Gamma_1}=\Gamma_1$
and $\overline{h\alpha}=\overline{h}\overline{\alpha}$ for all $h\in \Hecke$ and $\alpha\in M(\textbf{E}_J, L)$.
\end{defn}
An obvious induction on $\ell(w)$ shows that $\Gamma_w=T_w\Gamma_1$ for all $w\in \textbf{E}_J$.
\begin{defn}~(\cite[Definition 5.2]{Howlett})
 If $w \in W$ and $\textbf{E}_J = \{ u \in W \mid u \leqslant_L w \}$ is a $W$-graph ideal with
respect to some $J \subseteq S$ then we call $w$ a $W$-graph determining element.
\end{defn}

\textbf{Remark}. It has been verified in ~\cite[Section 5]{Howlett} that if $W$ is finite then $w_S$,
the maximal length element of $W$, is a $W$-graph determining element with respect to $\emptyset$, and
$d_J$ , the minimal length element of the left coset $w_SW_J$ , is a $W$-graph determining
element with respect to $J$ and also with respect to $\emptyset$.

\section{Duality theorem for $W$-graph ideals}

Let $(W,S)$ be a Coxeter group with weight function $L$ such that $L(s)\geqslant 0$ for all $s\in S$, $\Hecke$ be the corresponding Hecke algebra.
There exists an algebra map $\Phi: \Hecke\rightarrow \Hecke $ given by  $\Phi(q^{L(s)})=q^{L(s)}$ for all $s\in S$, and $\Phi(T_w)=\epsilon_w\overline{T_w}$, where the bar
is the standard involution in $\Hecke$. Further, $\Phi^2=Id$ and $\Phi$ commutes with the bar involution.
\subsection*{Duality theorem}
We now give an equivalent definition of a $W$-graph ideal, and the associated module is denoted by $\widetilde{M}(\textbf{E}_J, L)$.
 The following theorem essentially provides the duality between the two set ups.

\begin{thm-defn}
(I) With the above notation, let
the set $\textbf{E}_J$  be a $W$-graph ideal
with respect to $J (\subseteq S)$ and $L$, then the following hypotheses are satisfied.

(i) There exists an $A$-free $\Hecke$-module $\widetilde{M}(\textbf{E}_J, L)$ possessing an $A$-basis

 \[
 \widetilde{B} =\{ \widetilde{\Gamma}_w | w \in \textbf{E}_J \},
 \]
 for any $s\in S$ and any $w\in \textbf{E}_J $ we have
\begin{equation}\label{eq:4}
T_s\widetilde{\Gamma}_w=
\begin{cases}
\widetilde{\Gamma}_{sw}+(q^{L(s)}-q^{-L(s)})\widetilde{\Gamma}_w &\text{if $w\in \textbf{E}_{J,s}^-$},\\
\widetilde{\Gamma}_{sw}
&\text{if $w\in \textbf{E}_{J,s}^+$},\\
q^{L(s)}\widetilde{\Gamma}_w & \text{if $w\in \textbf{E}_{J,s}^{0,-}$ },\\
-q^{-L(s)} \widetilde{\Gamma}_w +\sum_{\substack{z \in \textbf{E}_J\\z<w}}\widetilde{r}^s_{z,w}\widetilde{\Gamma}_z& \text{if $w\in \textbf{E}_{J,s}^{0,+}$},
\end{cases}
\end{equation}
where $\widetilde{r}^s_{z,w}=\epsilon_z\epsilon_w\overline{r^s_{z,w} }\in q^{-L(s)}A_{<0}$.
(ii) The module $\widetilde{M}(\textbf{E}_J, L)$ admits an $A$-semilinear involution $\widetilde{\alpha}\mapsto \overline{\widetilde{\alpha}}$ satisfying $\overline{\widetilde{\Gamma}_1}=\widetilde{\Gamma}_1$
and $\overline{h\widetilde{\alpha}}=\overline{h}\overline{\widetilde{\alpha}}$ for all $h\in \Hecke$ and $\widetilde{\alpha}\in \widetilde{M}(\textbf{E}_J, L)$.

(II) There exists a unique map $\eta: M(\textbf{E}_J, L)\rightarrow\widetilde{M}(\textbf{E}_J, L)$ such that
\begin{align*}
(i) \eta(\Gamma_1)&=\widetilde{\Gamma}_1; \\
(ii) \eta(h \Gamma)&= \Phi(h)\eta(\Gamma),  \text{for all
$h\in \Hecke$ and $\Gamma\in M(\textbf{E}_J, L)$ }.
\end{align*}
( i.e., $\eta$ is $\Phi$-linear).
 Further, it has the following properties:

(a) $\eta$ commutes with the involution 
on $M(\textbf{E}_J, L)$ and $\widetilde{M}(\textbf{E}_J, L)$.

(b) $\eta$ is one-to-one onto and the inverse $\theta$ of $\eta$, satisfies properties (i) and (ii) of $\eta$.
\end{thm-defn}

\begin{proof}
For $w\in \textbf{E}_J$, define $\eta(\Gamma_w)=\epsilon_w\overline{\widetilde{\Gamma}_w}$. Extend $\eta$ to the whole
of $M(\textbf{E}_J, L)$ by $\Phi$-linearity. Let $s\in S$. Then we have,
\begin{equation}
\eta(T_s\Gamma_w)=
\begin{cases}
\eta[\Gamma_{sw}+(q^{L(s)}-q^{-L(s)})\Gamma_w] &\text{if $w\in \textbf{E}_{J,s}^-$},\\
\eta(\Gamma_{sw})
&\text{if $w\in \textbf{E}_{J,s}^+$},\\
\eta(-q^{-L(s)}\Gamma_w) & \text{if $w\in \textbf{E}_{J,s}^{0,-}$ },\\
\eta(q^{L(s)} \Gamma_w-\sum_{\substack{z \in \textbf{E}_J\\z<w}}r^s_{z,w}\Gamma_z)& \text{if $w\in \textbf{E}_{J,s}^{0,+}$},
\end{cases}
\end{equation}
which equals to
\begin{equation}
\begin{cases}
\epsilon_{sw}\overline{\widetilde{\Gamma}_{sw}}+(q^{L(s)}-q^{-L(s)})\epsilon_w\overline{\widetilde{\Gamma}_w} &\text{if $w\in \textbf{E}_{J,s}^-$},\\
\epsilon_{sw}\overline{\widetilde{\Gamma}_{sw}}
&\text{if $w\in \textbf{E}_{J,s}^+$},\\
-q^{-L(s)}\epsilon_w\overline{\widetilde{\Gamma}_w} & \text{if $w\in \textbf{E}_{J,s}^{0,-}$ },\\
q^{L(s)} \epsilon_w\overline{\widetilde{\Gamma}_w}-\sum_{\substack{z \in \textbf{E}_J\\z<w}}r^s_{z,w}\epsilon_z\overline{\widetilde{\Gamma}_z}& \text{if $w\in \textbf{E}_{J,s}^{0,+}$},
\end{cases}
\end{equation}
for some polynomials $r^s_{z,w}\in q^{L(s)}A_{>0}$.
On the other hand
\begin{align*}
\Phi(T_s)\eta(\Gamma_w)&=-\overline{T_s}\epsilon_w\overline{\widetilde{\Gamma}_w}\\
&=(-1)^{\ell(w)+1}\overline{T_s\widetilde{\Gamma}_w}\\
&=(-1)^{\ell(w)+1}
\begin{cases}
\overline{\widetilde{\Gamma}_{sw}+(q^{L(s)}-q^{-L(s)})\widetilde{\Gamma}_w} &\text{if $w\in \textbf{E}_{J,s}^-$},\\
\overline{\widetilde{\Gamma}_{sw}}
&\text{if $w\in \textbf{E}_{J,s}^+$},\\
\overline{q^{L(s)}\widetilde{\Gamma}_w} & \text{if $w\in \textbf{E}_{J,s}^{0,-}$ },\\
\overline{-q^{-L(s)} \widetilde{\Gamma}_w -\sum_{\substack{z \in \textbf{E}_J\\z<w}}\widetilde{r}^s_{z,w}\widetilde{\Gamma}_z}& \text{if $w\in \textbf{E}_{J,s}^{0,+}$},
\end{cases}
\end{align*}
It is easy to check that these two expressions give the same result, and this shows that
$\eta(T_s\Gamma_w)=\Phi(T_s)\eta(\Gamma_w)$. It is also easy to see that $\eta(h\Gamma_w)=\Phi(h)\eta(\Gamma_w)$
for all $h\in \Hecke$ and $\Gamma_w\in M(\textbf{E}_J, L)$.

If $\eta'$ is another map satisfying properties (i) and (ii), then
\begin{align*}
\eta'(\Gamma_w)=\eta'(T_w\Gamma_1)=\Phi(T_w)\widetilde{\Gamma}_1=\epsilon_w\overline{T_w}\widetilde{\Gamma}_1
=\epsilon_w\overline{T_w\widetilde{\Gamma}_1}=\epsilon_w\overline{\widetilde{\Gamma}_w}
\end{align*}
It is now clear that $\eta'=\eta$.

To prove statement (a), observe that for any $\Gamma\in M(\textbf{E}_J, L)$, there exists $h\in \Hecke$ such that
$\Gamma=h\Gamma_1$. Thus
\begin{align*}
\overline{\eta(\Gamma)}=\overline{\eta(h\Gamma_1)}=\overline{\Phi(h)\widetilde{\Gamma}_1}=\overline{\Phi(h)}\widetilde{\Gamma}_1
=\Phi(\overline{h})\widetilde{\Gamma}_1=\eta(\overline{h}\widetilde{\Gamma}_1)=\eta(\overline{\Gamma}).
\end{align*}
This proves (a).

We interchange the role of these two modules to obtain a map
\[
\theta: \widetilde{M}(\textbf{E}_J, L) \rightarrow  M(\textbf{E}_J, L)
\]
such that $\theta(\widetilde{\Gamma}_w )=\epsilon_w\overline{\Gamma_w}$. It is easy to check that $\theta$ and $\eta$ are inverses of each other. This proves (b)
\end{proof}
\begin{corollary}
If $R_{x,y}$ and $\widetilde{R}_{x, y}$ are the polynomials given by the formula
\begin{align*}
\overline{\Gamma_y}=\!\! \sum_{x \in \textbf{E}_J}R_{x,y} \Gamma_x ,
\overline{\widetilde{\Gamma}_y}=\!\! \sum_{x \in \textbf{E}_J}\widetilde{R}_{x,y} \widetilde{\Gamma}_x
\end{align*}
then
\[
\overline{R_{x, y}}=\epsilon_x\epsilon_y\widetilde{R}_{x,y}.
\]
\end{corollary}
\begin{proof} Apply the function $\eta$ to both the sides of the formula for $\overline{\Gamma_y}$ and
use the fact that $\eta$ commutes with the involution and then use the
formula for $\overline{\widetilde{\Gamma}_y}$. We omit the details.
\end{proof}
The above result can also be proved by the following recursive formulas.
\begin{lemma}\cite[Prop. 4.1]{Yin}
Let $x,\,y\in \textbf{E}_J$. If $s\in S$ is such that $y\in \textbf{E}_{J, s}^-$
then
\[
\postdisplaypenalty=10000 \advance\abovedisplayskip 0 pt minus 3
pt \advance\belowdisplayskip 0 pt minus 3 pt R_{x,y}=
\begin{cases}
R_{sx,sy}&\text{if $x\in \textbf{E}_{J,s}^-$},\\
R_{sx,sy}+(q^{-L(s)}-q^{L(s)})R_{x,sy}
&\text{if $x\in \textbf{E}_{J,s}^+$},\\
-q^{L(s)}R_{x,sy} & \text{if $x\in \textbf{E}_{J,s}^{0,-}$ },\\
q^{-L(s)}R_{x,sy}& \text{if $x\in \textbf{E}_{J,s}^{0,+}$}.
\end{cases}
\]
\end{lemma}
Similarly we have
\begin{lemma}
Let $x,\,y\in \textbf{E}_J$. If $s\in S$ is such that $y\in \textbf{E}_{J, s}^-$
then
\[
\postdisplaypenalty=10000 \advance\abovedisplayskip 0 pt minus 3
pt \advance\belowdisplayskip 0 pt minus 3 pt \widetilde{R}_{x,y}=
\begin{cases}
\widetilde{R}_{sx,sy}&\text{if $x\in \textbf{E}_{J,s}^-$},\\
\widetilde{R}_{sx,sy}+(q^{-L(s)}-q^{L(s)})\widetilde{R}_{x,sy}
&\text{if $x\in \textbf{E}_{J,s}^+$},\\
q^{-L(s)}\widetilde{R}_{x,sy} & \text{if $x\in \textbf{E}_{J,s}^{0,-}$ },\\
-q^{L(s)}\widetilde{R}_{x,sy}& \text{if $x\in \textbf{E}_{J,s}^{0,+}$}.
\end{cases}
\]
\end{lemma}
We have the further properties of $R_{x,y}$.
\begin{lemma}
If $y\in \textbf{E}_{J,s}^{0,-}$ then we have
\[
\postdisplaypenalty=10000 \advance\abovedisplayskip 0 pt minus 3
pt \advance\belowdisplayskip 0 pt minus 3 pt R_{x,y}=
\begin{cases}
-q^{-L(s)}R_{sx,y}&\text{if $x\in \textbf{E}_{J,s}^-$},\\
-q^{L(s)}R_{sx,y}
&\text{if $x\in \textbf{E}_{J,s}^+$}.
\end{cases}
\]

If $y\in \textbf{E}_{J,s}^{0,+}$ then we have
\[
\postdisplaypenalty=10000 \advance\abovedisplayskip 0 pt minus 3
pt \advance\belowdisplayskip 0 pt minus 3 pt R_{x,y}=
\begin{cases}
q^{L(s)}R_{sx,y}&\text{if $x\in \textbf{E}_{J,s}^-$},\\
q^{-L(s)}R_{sx,y}
&\text{if $x\in \textbf{E}_{J,s}^+$}.
\end{cases}
\]

\end{lemma}
\begin{proof}
If $y\in \textbf{E}_{J,s}^{0,-}$ then
\[
T_s\Gamma_y=-q^{-L(s)}\Gamma_y
\]
Applying involution bar on both sides. On the left hand side we have
\[
\overline{T_s\Gamma_y}=\overline{T_s}\overline{\Gamma_y}=[T_s+(q^{-L(s)}-q^{L(s)}]\sum_{x\in \textbf{E}_J}R_{x, y}\Gamma_x.
\]
while the right hand side is $\overline{-q^{-L(s)}\Gamma_y}=-q^{L(s)}\sum_{x\in \textbf{E}_J}R_{x, y}\Gamma_x$.

Comparing the coefficients of $\Gamma_x$ in the two expressions, we get the result.
The proof for the case $y\in \textbf{E}_{J,s}^{0,+}$ is similar with the above.
\end{proof}

\subsection*{Dual bases $\textbf{C}$ and $\textbf{C}'$}
Recall \cite[Th.4.4]{Yin}
 that the invariants in  $M(\textbf{E}_J, L)$ (respectively $\widetilde{M}(\textbf{E}_J, L)$) form a
free $A$-module with a basis $\{\,\textbf{C}_w\mid w\in \textbf{E}_J\,\}$
(respectively $\{\,\widetilde{\textbf{C}}_w\mid w\in \textbf{E}_J\,\}$   ), where
$\textbf{C}_w\,=\sum\limits_{y \in \textbf{E}_J}\!\!P_{y,w}\Gamma_y$
and
$\widetilde{\textbf{C}}_w=\sum\limits_{y \in \textbf{E}_J}\!\!\widetilde{P}_{y,w}\widetilde{\Gamma}_y$.

Using the map $\theta$, we obtain a dual basis $\{\,\textbf{C}'_w\mid w\in \textbf{E}_J\,\}$  for the
invariants in $M(\textbf{E}_J, L)$ . Analogously, using the map $\eta$ we obtain the dual basis $\{\,\widetilde{\textbf{C}}'_w\mid w\in \textbf{E}_J\,\}$ for the
invariants in $\widetilde{M}(\textbf{E}_J, L)$.

 More precisely, we have:

\begin{proposition}
Let $\textbf{C}'_w=\theta(\widetilde{\textbf{C}}_w)$, $\widetilde{\textbf{C}}'_w=\eta(\textbf{C}_w)$. Then

(a)
The $\Hecke$-module $M(\textbf{E}_J, L)$ has a unique basis
$\{\,\textbf{C}'_w\mid w\in \textbf{E}_J\,\}$ such that $\overline{\textbf{C}'_w}=\textbf{C}'_w$ for all
$w\in \textbf{E}_J$, and $\textbf{C}'_w\,=\sum\limits_{y \in \textbf{E}_J}\!\!\epsilon_y\overline{\widetilde{P}_{y,w}}\Gamma_y.$
for some elements $\widetilde{P}_{y,w}\in A_{\geqslant 0}$ with the following
properties\/\textup{:}
\begin{itemize}
\item[(a1)] $\widetilde{P}_{y,w}=0$ if $y\nleqslant w$\textup{;}
\item[(a2)]
$\widetilde{P}_{w,w}=1$;
\item[(a3)] $\widetilde{P}_{y,w}$ has zero constant term if
$y\neq w$ and
\[
\overline{\widetilde{P}_{y, w}}-\widetilde{P}_{y,w}=\!\!\! \sum_{\substack{y<x\leqslant w\\ x\in \textbf{E}_J}}\overline{\widetilde{R}_{y,x}}\widetilde{P}_{x, w} \text{for any $y<w$}.
\]
(b) Analogously, the module $\widetilde{M}(\textbf{E}_J, L)$ has another basis $\{\,\widetilde{\textbf{C}}'_w\mid w\in \textbf{E}_J\,\}$, where
$\widetilde{\textbf{C}}'_w=\sum\limits_{y \in \textbf{E}_J}\!\!\epsilon_y\overline{P_{y,w}}\widetilde{\Gamma}_y$.
\end{itemize}
\end{proposition}
\begin{proof}
\begin{align*}
\textbf{C}'_w\,&=\theta({\sum\limits_{y \in \textbf{E}_J}\!\!\widetilde{P}_{y,w}\widetilde{\Gamma}_y})
=\sum\limits_{y \in \textbf{E}_J}\!\!\epsilon_y\widetilde{P}_{y,w}\overline{\widetilde{\Gamma}_y}
\end{align*}
Hence, $\overline{\textbf{C}'_w}=\overline{\theta(\widetilde{\textbf{C}}_w)}
=\theta(\overline{\widetilde{\textbf{C}}_w})=\theta(\widetilde{\textbf{C}_w})=\textbf{C}'_w$ and the result follows.
\end{proof}

\subsection*{Inversion}
For $y, w\in \textbf{E}_J$, we write the matrix $P=(P_{y, w})$ , where $P_{y, w}$ are $\textbf{E}_J$-relative Kazhdan-Lusztig polynomials.
The formula for $\textbf{C}_w$ in ~\cite[Th.4.4]{Yin} may be written as
\[
\textbf{C}_w=\Gamma_w+\sum_{y\in \textbf{E}_J}P_{y, w}\Gamma_y
\]
and inverting this gives
\[
\Gamma_w=\textbf{C}_w+\sum_{y\in \textbf{E}_J}Q_{y, w}\textbf{C}_y
\]
where the elements $Q_{y, w}$ (defined whenever $y<w$) are given recursively by
\begin{equation}\label{ }
Q_{y, w}=-P_{y, w}-\sum_{z\in \textbf{E}_J\mid y<z<w}Q_{y, z}P_{z, w}
\end{equation}
A $\textbf{E}_J$-chain is a sequence $\zeta: z_0<z_1<\cdots<z_n (n\geqslant 1)$ of elements in
$\textbf{E}_J$, we set $\ell(\zeta)=n$ and $P_\zeta=P_{z_0, z_1}P_{z_1, z_2}\cdots P_{z_{n-1}, z_n}$.
$z_0$ is called the initial element of $\zeta$ and $z_n$ is called the final element of  $\zeta$.
For $y<w$, let $\tau(y, w)$ denote the set of all $\textbf{E}_J$-chains with $y$ as the initial element and $w$ as the final element.

The following results are motivated by Lusztig~\cite[Ch. 10]{Lusztig2}. For the sake of completeness we attach the proofs.
\begin{proposition}For any $y, w\in \textbf{E}_J$ we have
 \[
 Q_{y, w}=\sum_{\zeta\in \tau(y, w)}(-1)^{\ell(\zeta)} P_\zeta
 \]
 We have
 $Q_{y,w}\in A_{\geqslant 0}$ with the following
properties\/\textup{:}
\begin{itemize}
\item[(a1)] $Q_{y,w}=0$ if $y\nleqslant w$\textup{;}
\item[(a2)]
$Q_{w,w}=1$;
\end{itemize}
\end{proposition}
\begin{proof}
 If $\ell(w)-\ell(y)=1$, by Eq.(7) we have $Q_{y, w}=-P_{y, w}$. The statement is true.
 Applying induction on $\ell(w)-\ell(y)\geqslant 1$. For any $z\in \textbf{E}_J, y<z<w$, in the sum of Eq.(7)
 we use the induction hypothesis.
 \begin{align*}
 Q_{y, z}=\sum_{\zeta'\in \tau(y, z)}(-1)^{\ell(\zeta')} P_{\zeta'}
 \end{align*}
 We have
 \begin{align*}
 Q_{y, w}&=-P_{y, w}-\sum_{\zeta'\in \tau(y, z)}(-1)^{\ell(\zeta')} P_\zeta P_{z, w}\\
 &=\sum_{\zeta\in \tau(y, w)}(-1)^{\ell(\zeta)} P_\zeta
\end{align*}
where the sequence $\zeta=(y, w) (\in \tau(y, w))$ is with $\ell(\zeta)=1$ and $(\zeta',w) (\in \tau(y, w))$ with the length $\ell(\zeta')+1$.
 The listed properties of $Q's$ are by Eq.(7). The result is proved.
\end{proof}

We define
 \[
 Q'_{y, w}=sgn(y)sgn(w)Q_{y, w}
 \]
 \begin{proposition}
 For any $y, w\in \textbf{E}_J$ we have $\overline{Q'_{y, w}}=\sum_{z;y\leqslant_L z\leqslant_L w}Q'_{y, z}\overline{\widetilde{R}_{z, w}}$
 \end{proposition}
 \begin{proof}
 The triangular matrices $Q=(Q_{y, w}), P=(P_{y,w}), R=(R_{y, w})$ are related by
 \[
 PQ=QP=1, \overline{P}=\overline{R}P, \overline{R}R=R\overline{R}=1
 \]
 where the bar involution over a matrix is the matrix obtained by applying $\bar{}$ to each entry.
 We deduce that
 \[
 QP=1=\overline{Q}\overline{P}=\overline{Q}\overline{R}P
 \]
 Multiplying on the right by $Q$ and using the fact $PQ=1$ we deduce $Q=\overline{Q}\overline{R}$. Multiplying on the right by $R$ gives
 \[
 \overline{Q}=QR
 \]
 Let $S$ be the matrix whose $(y, w)$-entry is $sgn(y)\delta_{y,w}$. We have $S^2=1$.  Note that $Q'=SQS$. By Corollary 3.2 we have $\overline{R}=S\widetilde{R}S$.
 Hence
 \[
 \overline{Q'}=\overline{SQS}=S(QR)S=SQS\cdot SRS=Q'\overline{\widetilde{R}}
 \]
 The result follows.
 \end{proof}

\section{$W$-graphs for the modules $\hat{M}$ and $\hat{\widetilde{M}}$ }
 Denote by $M:=M(\textbf{E}_J, L)$ and $\widetilde{M}:==\widetilde{M}(\textbf{E}_J, L)$.
Let $\hat{M}:=Hom_A(M, A)$ and $\hat{\widetilde{M}}:=Hom_A(\widetilde{M}, A)$.

Define an left $\Hecke$-module structure on
$\hat{M}$ by
\[
hf(m)=f(hm) \text{(with $f\in \hat{M}, m\in M, h\in \Hecke$ )}.
\]
 We define a bar operator $\hat{M}\mapsto \hat{M}$
by $\overline{f}(m)=\overline{f(\overline{m})}$ (with $f\in \hat{M}, m\in M$ ); in $\overline{f(\overline{m})}$ the lower bar is that of $M$
and the upper bar is that of $A$.
\[
\overline{h\cdot f}(m)=\overline{hf(\overline{m})}=\overline{f(h\overline{m})}
=\overline{f}(\overline{h\overline{m}})=\overline{f}(\overline{h}m)={\overline{h}} \cdot \overline{f}(m).
\]
Hence we have $\overline{h\cdot f}=\overline{h}\cdot\overline{f}$ for $f\in \hat{M}, h\in \Hecke$.

In the following contexts we focus on the module $\hat{M}$, and usually omit the analogous details for $\hat{\widetilde{M}}$.

If $P$ is a property we set $\delta_P=1$ if $P$ is true and $\delta_P=0$ if $P$ is false. We write $\delta_{x, y}$ instead
of $\delta_{x=y}$.

\subsection*{The basis of $\hat{M}$ } 
We firstly introduce two bases for the module $\hat{M}$.
For any $z\in \textbf{E}_J$ we define
$\hat{\Gamma}_z\in \hat{M}$ by $\hat{\Gamma}_z(\Gamma_w)=\delta_{z, w}$ for any $w\in \textbf{E}_J$. Then
$\hat{B}=:\{\hat{\Gamma}_z; z\in \textbf{E}_J\}$ is an $A$-basis of $\hat{M}$.

Further, for any $z\in \textbf{E}_J$ we define
$D_z\in \hat{M}$ by $D_z(\textbf{C}_w)=\delta_{z, w}$ for any $w\in \textbf{E}_J$. Then
$D:=\{D_z; z\in \textbf{E}_J\}$ is an $A$-basis of $\hat{M}$.

 Obviously we have
\[
D_z=\sum_{y\in \textbf{E}_J, z<y}Q_{z, y}\hat{\Gamma}_y.
\]

An equivalent definition of the basis element $D_w\in \hat{M}$ is
\[
D_z(\Gamma_y)=Q_{z, y}
\]
for all $y\in \textbf{E}_J$. In fact, we have
\[
D_z(\textbf{C}_w)=D_z\sum_{y\in \textbf{E}_J}P_{y,w}\Gamma_y)=\sum_{y\in \textbf{E}_J}Q_{z,y}P_{y, w}=\delta_{z, w}
\]

\begin{lemma}For any $y\in \textbf{E}_J$ we have
\[
\overline{\hat{\Gamma}_y}=\sum_{w\in \textbf{E}_J, y\leqslant w}\overline{R_{y, w}}\hat{\Gamma}_w.
\]
\end{lemma}
\begin{proof}
For any $x\in \textbf{E}_J$ we have
\begin{align*}
\overline{\hat{\Gamma}_y}(\Gamma_x)&=\overline{\hat{\Gamma}_y(\overline{\Gamma_x})}\\
&=\overline{\hat{\Gamma}_y(\sum_{x'\in \textbf{E}_J,x'\leqslant x}R_{x', x}\Gamma_{x'})}=\overline{\delta_{y\leqslant x}R_{y, x}}=\delta_{y\leqslant x}\overline{R_{y, x}}\\
&=\sum_{w\in \textbf{E}_J, y\leqslant w}\overline{R_{y, w}}\hat{\Gamma}_w(\Gamma_x)
\end{align*}
\end{proof}

\begin{thm}\cite[Th.4.7 ]{Yin}\label{thm:4}
The basis elements $\{\textbf{C}_v\mid v\in \textbf{E}_J\}$ give the module $M(\textbf{E}_J, L)$
the structure of a $W\!$-graph module such that
\begin{equation}\label{eq:11}
T_s \textbf{C}_v=
\begin{cases}
q^{L(s)} \textbf{C}_v+ \textbf{C}_{sv}+\!\!\sum\limits_{z\in \textbf{E}_J,sz<z<v}\!\! m_{z,v}^s \textbf{C}_z&
\text{if $s\in SA(v)$,}\\[-5 pt]
-q^{-L(s)}\textbf{C}_v& \text{if $s\in D(v)$},\\[5 pt]
q^{L(s)} \textbf{C}_v + \!\!\sum\limits_{z\in \textbf{E}_J,sz<z<v}\!\! m_{z,v}^s \textbf{C}_z&
\text{if $s\in WA(v)$}.
\end{cases}
\advance\belowdisplayskip5 pt
\end{equation}
\end{thm}
\begin{thm}
The $\Hecke$-module $\hat{M}(\textbf{E}_J, L)$ has a unique basis
$\{\,D_z\mid z\in \textbf{E}_J\,\}$ such that $\overline{D_z}=D_z$ for all
$z\in \textbf{E}_J$, and $D_z\,=\sum\limits_{y \in \textbf{E}_J}\!\!Q_{z,y}\hat{\Gamma}_y.$
for some elements $Q_{z, y}\in A_{\geqslant 0}$ with the following
properties\/\textup{:}
\begin{itemize}
\item[(a1)] $Q_{z,y}=0$ if $z\nleqslant y$\textup{;}
\item[(a2)]
$P_{z,z}=1$;
\item[(a3)] $Q_{z,y}$ has zero constant term if
$z\neq y$ and
\[
Q_{z, y}-\overline{Q_{z,y}}=\!\!\! \sum_{\substack{z\leqslant x<y\\ x\in \textbf{E}_J}}\overline{Q_{z,x}}R_{x, y} \text{for any $z<y$}.
\]
\end{itemize}
\end{thm}
The proof is very similar to that of \cite[Th. 5.2]{Lusztig2} or \cite[Section 2]{Lusztig1}. It uses induction on $\ell(w)-\ell(y)$,
the equation $\overline{Q}=QR$ in Proposition 3.8, and the fact: If $f=\!\!\!\sum_{\substack{z\leqslant x<y\\ y\in \textbf{E}_J}}\overline{Q_{z,x}}R_{x, y}$ then
$\overline{f}=-f$. We omit further details of the proof.

The (left) ascent set of $z\in \textbf{E}_J$ is
\[
A(z)=\{s\in S\mid z\in {\textbf{E}_{J,s}^+}\cup {\textbf{E}_{J,s}^{0,+}}\}
\]

\begin{thm}
Let $s\in S$ and assume that $L(s)>0$. 
 The basis elements
 \[
 \{D_z\mid z\in \textbf{E}_J\}
 \]
  give $\hat{M}$
the structure of a $W\!$-graph module such that
\begin{equation}
T_s D_z=
\begin{cases}
-q^{-L(s)}D_z+ D_{sz}+\!\!\sum\limits_{z<u,s\in A(u)}\!\! m_{z,u}^s D_u&
\text{if $s\in SD(z)$,}\\[3 pt]
q^{L(s)}D_z&
\text{if $s\in A(z)$,}\\[3 pt]
-q^{-L(s)}D_z+\!\!\sum\limits_{z<u, s\in A(u)}\!\! m_{z,u}^s D_u&
\text{if $s\in WD(z)$,}
\end{cases}
\advance\belowdisplayskip5 pt
\end{equation}
\end{thm}
\begin{proof}
In the case $s\in SD(z)$, $T_s D_z(\textbf{C}_w)=D_z(T_s\textbf{C}_w)$ gives
\begin{equation*}
T_s D_z(\textbf{C}_w)=
\begin{cases}
D_z(q^{L(s)} \textbf{C}_w+ \textbf{C}_{sw}+\!\!\sum\limits_{x\in \textbf{E}_J,sx<x<w}\!\! m_{x,w}^s \textbf{C}_x)
&\text{if $s\in SA(w)$,}\\[3 pt]

D_z(-q^{-L(s)} \textbf{C}_w)&\text{if $s\in D(w)$,}\\[3 pt]

D_z(q^{L(s)} \textbf{C}_w+\!\!\sum\limits_{x\in \textbf{E}_J,sx<x<w}\!\! m_{x,w}^s \textbf{C}_x)
&\text{if $s\in WA(w)$,}
\end{cases}\\
\end{equation*}

\begin{equation*}
=
\begin{cases}
\delta_{z,sw}+\!\!\sum\limits_{x\in \textbf{E}_J,sx<x<w}\!\! m_{x,w}^s \delta_{z,x}&\text{if $s\in SA(w)$,}\\[3 pt]
-q^{-L(s)}\delta_{z,w}&\text{if $s\in SD(w)$,}\\[3 pt]
0&\text{if $s\in WD(w)$,}\\[3 pt]
\!\!\sum\limits_{x\in \textbf{E}_J,sx<x<w}\!\! m_{x,w}^s \delta_{z,x}&\text{if $s\in WA(w)$,}
\end{cases}
\end{equation*}
\begin{equation*}
=
\begin{cases}
\delta_{z,sw}+\!\! m_{z,w}^s \delta_{z<w}&\text{if $s\in SA(w)$,}\\[3 pt]
-q^{-L(s)}\delta_{z,w}&\text{if $s\in SD(w)$,}\\[3 pt]
0&\text{if $s\in WD(w)$,}\\[3 pt]
 m_{z,w}^s \delta_{z<w}&\text{if $s\in WA(w)$,}
\end{cases}
\end{equation*}
\begin{equation*}
=
\begin{cases}
(D_{sz}+\!\!\sum\limits_{z<u, u\in \textbf{E}_{J,s}^+}\!\! m_{z,u}^s D_u)(\textbf{C}_w)&\text{if $s\in SA(w)$,}\\[3 pt]
-q^{-L(s)}D_z(\textbf{C}_w)&\text{if $s\in SD(w)$,}\\[3 pt]
0&\text{if $s\in WD(w)$,}\\[3 pt]
 \sum\limits_{z<u, u\in \textbf{E}_{J,s}^{0,+}}\!\! m_{z,u}^s D_u(\textbf{C}_w)  &\text{if $s\in WA(w)$,}
\end{cases}
\advance\belowdisplayskip5 pt
\end{equation*}
Hence, we obtain
\[
T_sD_z(\textbf{C}_w)=(-q^{-L(s)}D_z+ D_{sz}+\!\!\sum\limits_{z<u,s\in A(u)}\!\! m_{z,u}^s D_u)(\textbf{C}_w)
\]
foe all $w\in \textbf{E}_J$. The desired formula follows in this case.

In other cases the computation is similar with the above, we omit the details.
\end{proof}
The following is by \cite[Prop.4.8]{Yin}.
\begin{corollary}
For $s\in S$ with $L(s)=0$, $z\in \textbf{E}_J$, we have
\[
T_sD_z=
\begin{cases}
D_{sz}&\text{if $s\in SD(z)$ or $s\in SA(z)$,}\\[3 pt]
-D_z&\text{if $s\in WD(z)$,}\\[3 pt]
D_z&\text{if $s\in WA(z)$,}
\end{cases}
\]
\end{corollary}
\subsection*{The $D'$-basis for $M$}
\begin{thm}
The $\Hecke$-module $\hat{M}(\textbf{E}_J, L)$ has a unique basis
$\{\,{D}'_z\mid z\in \textbf{E}_J\,\}$ such that $\overline{{D}'_z}={D}'_z$ for all
$z\in \textbf{E}_J$, and ${D}'_z\,=\sum\limits_{y \in \textbf{E}_J}\!\!\epsilon_y\overline{\widetilde{Q}_{z,y}}\hat{\Gamma}_y.$
where $\widetilde{Q}_{z, y}\in A_{\geqslant 0}$, are the analogous elements in the case of $\widetilde{M}$.
\begin{equation}
T_s {D}'_z=
\begin{cases}
q^{L(s)}{D}'_z+ {D}'_{sz}+\!\!\sum\limits_{z<u,s\in A(u)}\!\! m_{z,u}^s {D}'_u&
\text{if $s\in SD(z)$,}\\[3 pt]
-q^{-L(s)}{D}'_z&
\text{if $s\in A(z)$,}\\[3 pt]
q^{L(s)}{D}'_z+\!\!\sum\limits_{z<u, s\in A(u)}\!\! m_{z,u}^s {D}'_u&
\text{if $s\in WD(z)$,}
\end{cases}
\advance\belowdisplayskip5 pt
\end{equation}

\end{thm}
 For the $\Hecke$-module $M(\textbf{E}_J, L)$, two pairs of dual bases $\textbf{C}, \textbf{C}'$ and $D, D'$ give the structures of the "full $W$-graphs".

\subsection*{The module $\hat{M}(D_J, L)$}
Set $\textbf{E}_J:=D_J$. If $D_J$ is regarded as a $W$-graph ideal with respect to $\emptyset$ (see Deodhar's construction in Section 6), we have
\begin{lemma}
The modules $\hat{M}(D_J, L)$ and $M(D_J, L)$ are identical.
\end{lemma}
\begin{proof}
For any basis element $\hat{\Gamma}_w$ of $\hat{M}(D_J, L)$ and element $\Gamma_y$ of $M(D_J, L)$,  we have
\begin{align*}
T_s\hat{\Gamma}_w(\Gamma_y)&=\hat{\Gamma}_w(T_s\Gamma_y)\\
&=\delta_{y\in D_{J,s}^-}\delta_{w, sy}+(q^{L(s)}-q^{-L(s)})\delta_{y\in D_{J,s}^-}\delta_{w,y}+\delta_{y\in D_{J,s}^+}\delta_{w,sy}\\
&\hspace{.7 cm}
+q^{L(s)}\delta_{y\in D_{J,s}^0}\delta_{w, y}\\
&=\delta_{w\in D_{J,s}^+}\delta_{sw, y}+(q^{L(s)}-q^{-L(s)})\delta_{w\in D_{J,s}^-}\delta_{w, y}+\delta_{w\in D_{J,s}^-}\delta_{sw, y}\\
&\hspace{.7 cm}
+q^{L(s)}\delta_{w\in D_{J,s}^0}\delta_{w, y}\\
&=(\delta_{w\in D_{J,s}^+}\hat{\Gamma}_{sw}+(q^{L(s)}-q^{-L(s)})\delta_{w\in D_{J,s}^-} \hat{\Gamma}_{w}+\delta_{w\in D_{J,s}^-}\hat{\Gamma}_{sw}\\
&\hspace{.7 cm}
+q^{L(s)}\delta_{w\in D_{J,s}^0}\hat{\Gamma}_{w})(\Gamma_y)
\end{align*}
hence we have
\[
T_s\hat{\Gamma}_w=
\begin{cases}
\hat{\Gamma}_{sw}&\text{if $w\in D_{J,s}^+$,}\\[3 pt]
\hat{\Gamma}_{sw}+(q^{L(s)}-q^{-L(s)})\hat{\Gamma}_{w}&\text{if $w\in D_{J,s}^-$,}\\[3 pt]
q^{L(s)}\hat{\Gamma}_{sw}&\text{if $w\in D_{J,s}^0$,}
\end{cases}
\]
The result follows.
\end{proof}

\begin{corollary}
 The $\Hecke$-module $M(D_J, L)$ has basis $\{D_z\mid z\in D_J \}$, where $D_z=\sum_{y\in D_J, z<y}Q_{z, y}\Gamma_y$.
 This basis gives the structure of $W$-graph module such that
\begin{equation*}
T_s D_z=
\begin{cases}
-q^{-L(s)}D_z+ D_{sz}+\!\!\sum\limits_{z<u,u\in {D_{J,s}^+}\cup{D_{J,s}^{0}}}\!\! m_{z,u}^s D_u&
\text{if $z\in D_{J. s}^-$,}\\[3 pt]
q^{L(s)}D_z&
\text{if $z\in D_{J, s}^+\cup D_{J, s}^0$,}
\end{cases}
\advance\belowdisplayskip5 pt
\end{equation*}
\end{corollary}

\section{In the case $W$ is finite}
Let $(W,S)$ be a finite Coxeter system and $w_0$ be the longest element in $W$.
Define the function $\pi\colon W\to W$ by $\pi(w)=w_0ww_0$, it satisfies
$\pi(S) = S$ and it extends to a $C$-algebra isomorphism $\pi \colon C[W] \longmapsto C[W]$.
 We denote by $s_0=\pi(s)$.
 For $s\in S$
we have $\ell(w_0) = \ell(w_0s)+\ell(s) = \ell(\pi(s))+\ell(\pi(s)w_0)$, hence
\[
L(w_0) = L(w_0s)+L(s) =
L(\pi(s)) + L(\pi(s)w_0) = L(\pi(s)) + L(w_0s)
\]
 so that $L(\pi(s)) = L(s)$. It follows that
$L(\pi(w)) = L(w)$ for all $w \in W$ and that we have an $A$-algebra automorphism
$\pi: \Hecke\longmapsto \Hecke$ where $\pi(T_w) = T_{\pi(w)}$ for any $w \in W$.

\begin{lemma}
The $\Hecke$-modules $M$ and $\widetilde{M}$ have basis $\Gamma^\pi=\{T_{w_0}\overline{\Gamma_w} \mid w\in \textbf{E}_J \}$
and $\widetilde{\Gamma}^{\pi}=\{T_{w_0}\overline{\widetilde{\Gamma}_w} \mid w\in \textbf{E}_J \}$ respectively. Moreover we have
$\eta(T_{w_0}\overline{\Gamma_w})=\epsilon_{w_0w}\overline{T_{w_0}\overline{\widetilde{\Gamma}_w}}$.
\end{lemma}
\begin{proof}
 Since the involution is square $1$ and $T_{w_0}$ is invertible in $\Hecke$, the statement follows.
 Furthermore
 \[
 \eta(T_{w_0}\overline{\Gamma_w})=\Phi(T_{w_0})\eta(\overline{\Gamma_w})=\epsilon_{w_0}\overline{T_{w_0}}\epsilon_w\widetilde{\Gamma}_w
 =\epsilon_{w_0w}\overline{T_{w_0}\overline{\widetilde{\Gamma}_w}}.
 \]
\end{proof}
In the following, for the sake of convenience we primarily focus on the module $M$ and omit the analogous details for $\widetilde{M}$,
unless it is needed.
For any $w\in \textbf{E}_J $ we denote by $w':=w_0w$ and $\Gamma^{\pi}_{w'}:=T_{w_0}\overline{\Gamma_w} (\in M(\textbf{E}_J, L))$.

 \emph{Remark}   Generally $w_0\textbf{E}_J\neq \textbf{E}_J$. We emphasize that, in the following contexts, the set $w_0\textbf{E}_J$ will be
 just used as the index set for the objects involved.

Direct computation gives the following multiplication rules for the basis ${\Gamma}^{\pi}$.
\[
T_{s_0}\Gamma^{\pi}_{w'}=
\begin{cases}
\Gamma^{\pi}_{s_0w'}+(q^{L(s)}-q^{-L(s)})\Gamma^\pi_{w'} &\text{if $w\in \textbf{E}_{J,s}^+$},\\
\Gamma^{\pi}_{s_0w'}
&\text{if $w\in \textbf{E}_{J,s}^-$},\\
-q^{-L(s)}\Gamma^{\pi}_{w'} & \text{if $w\in \textbf{E}_{J,s}^{0,-}$ },\\
q^{L(s)} \Gamma^\pi_{w'}-\sum_{\substack{z \in \textbf{E}_J\\z<w}}r^{s_0}_{w',z'}\Gamma^\pi_{z'}& \text{if $w\in \textbf{E}_{J,s}^{0,+}$},
\end{cases}
\]where $r^{s_0}_{w',z'}=\overline{r^{s}_{z,w}}\in q^{-L(s)}A_{<0}$.

\begin{lemma} For any $y'\in w_0\textbf{E}_J$ there exist coefficients $R^\pi_{x',y'}\in A$, defined for $x'\in w_0\textbf{E}_J$
and $x'<y'$, such that $\overline{\Gamma^{\pi}_{y'}}=\sum_{x'\in w_0\textbf{E}_J} R^\pi_{x',y'}\Gamma^{\pi}_{x'}$.
If $R^\pi_{x',y'}\neq 0$ then $x'\leqslant y'$; particularly $R^\pi_{y',y'}=1$.
\end{lemma}
The proof is trivial.

We have further properties of $R^\pi_{x',y'}$.
\begin{lemma}
If $y'\in w_0\textbf{E}_{J,s}^{0,-}$ then we have
\[
\postdisplaypenalty=10000 \advance\abovedisplayskip 0 pt minus 3
pt \advance\belowdisplayskip 0 pt minus 3 pt R^\pi_{x',y'}=
\begin{cases}
-q^{L(s_0)}R^\pi_{s_0x',y'}&\text{if $x'\in w_0\textbf{E}_{J,s}^-$},\\
-q^{-L(s_0)}R^\pi_{s_0x',y'}
&\text{if $x'\in w_0\textbf{E}_{J,s}^+$}.
\end{cases}
\]

If $y'\in w_0\textbf{E}_{J,s}^{0,+}$ then we have
\[
\postdisplaypenalty=10000 \advance\abovedisplayskip 0 pt minus 3
pt \advance\belowdisplayskip 0 pt minus 3 pt R^\pi_{x',y'}=
\begin{cases}
q^{-L(s_0)}R^\pi_{s_0x',y'}&\text{if $x'\in w_0\textbf{E}_{J,s}^-$},\\
q^{L(s_0)}R^\pi_{s_0x',y'}
&\text{if $x'\in w_0\textbf{E}_{J,s}^+$}.
\end{cases}
\]
\end{lemma}
\begin{proof}
The proof is similar with that of Lemma 3.5.
\end{proof}


\subsection{The bases $\textbf{C}^\pi$ for $M$ }
The elements $R^\pi_{w', y'}$ , where $w', y'\in w_0\textbf{E}_J$, lead to the construction of another set of elements $P^\pi_{w',y'}$
and the following basis of $M(\textbf{E}_J, L)$ .

\begin{thm}\label{thm:2}
The $\Hecke$-module $M(\textbf{E}_J, L)$ has a unique basis
$\{\,\textbf{C}^\pi_{y'}\mid y'\in w_0\textbf{E}_J\,\}$ such that $\overline{\textbf{C}^\pi_{y'}}=\textbf{C}^\pi_{y'}$ for all
$y\in w_0\textbf{E}_J$, and $\textbf{C}^\pi_{y'}\,=\sum\limits_{w' \in w_0\textbf{E}_J}\!\!P^\pi_{w',y'}\Gamma^\pi_{w'}.$
for some elements $P^\pi_{w',y'}\in A_{\geqslant 0}$ with the following
properties\/\textup{:}
\begin{itemize}
\item[(a1)] $P^\pi_{w',y'}=0$ if $w'\nleqslant y'$\textup{;}
\item[(a2)]
$P^\pi_{y',y'}=1$;
\item[(a3)] $P^\pi_{w',y'}$ has zero constant term if
$y'\neq w'$ and
\[
\overline{P^\pi_{w', y'}}-P^\pi_{w',y'}=\!\!\! \sum_{\substack{w'<x'\leqslant y'\\ x'\in w_0\textbf{E}_J}}\overline{R^\pi_{w',x'}}P^\pi_{x', y'} \text{for any $w'<y'$}.
\]
\end{itemize}
\end{thm}
The proof is very similar to that of \cite[Th. 5.2]{Lusztig2} or \cite[Section 2]{Lusztig1}. It uses induction on $\ell(w')-\ell(y')$,
and the fact:
\[
\text{
If $f=\!\!\! \sum_{\substack{w'<x'\leqslant y'\\ x'\in w_0\textbf{E}_J}}\overline{R^\pi_{w',x'}}P^\pi_{x', y'}$ then
$\overline{f}=-f$.}
\]
We omit further details of the proof.

\begin{lemma} For $y, w\in \textbf{E}_J$. We have
(i) $y\leqslant_L w \Longleftrightarrow w'\leqslant_L y'$;

(ii) $R^\pi_{w', y'}=R_{y, w}$; $\widetilde{R}^\pi_{w', y'}=\widetilde{R}_{y, w}$;

(iii) for any $w', y'\in w_0\textbf{E}_J$ and $w'<y'$ we have
\begin{align*}
\overline{P^\pi_{w', y'}}=\!\!\! \sum_{\substack{w'\leqslant x'\leqslant y'\\ x'\in w_0\textbf{E}_J}}P^\pi_{x', y'}\overline{R^\pi_{x,w}} .\\
\overline{\widetilde{P}^\pi_{w', y'}}=\!\!\! \sum_{\substack{w'\leqslant x'\leqslant y'\\ x'\in w_0\textbf{E}_J}}\widetilde{P}^\pi_{x', y'}\overline{\widetilde{R}^\pi_{x,w}} 
\end{align*}
\end{lemma}
\begin{proof}(a) is obvious.
We prove (b) by induction on $\ell(w)$. If $\ell(w)=0$ then $w=1$. We have $R_{y, 1}=\delta_{y, 1}$. Now $R^\pi_{w_0, w_0y}=0$ unless $w_0\leqslant_L w_0y$. On the other hand we have $w_0y\leqslant_L w_0$. Hence $R^\pi_{w_0w, w_0y}=0$ unless $w_0y=w_0$, that is $y=1$
in which case it is $1$. The desired equality holds when $\ell(w)=0$. Assume that $\ell(w)\geqslant 1$. We can find
$s\in S$ such that $sw<w$. The proof of the following cases (a) and (b) is similar with Lusztig....

In the case (a) $y\in \textbf{E}_{J,s}^-$. By the induction hypothesis we have
\[
R_{y, w}=R_{sy, sw}=R^\pi_{w_0sw, w_0sy}=R^\pi_{s_0w_0w, s_0w_0y}=R^\pi_{w_0w, w_0y}
\]

In the case (b) $y\in \textbf{E}_{J,s}^+$. Using Lemma 3.3, by the induction hypothesis we have
\begin{align*}
R_{y, w}&=R_{sy, sw}+(q^{-L(s)}-q^{L(s)})R_{y, sw}\\
&=R^\pi_{w_0sw, w_0sy}+(q^{-L(s_0)}-q^{L(s_0)})R^\pi_{w_0sw, w_0y}\\
&=R^\pi_{s_0w', s_0y'}+(q^{-L(s_0)}-q^{L(s_0)})R^\pi_{s_0w', y'}\\
&=R^\pi_{s_0w', s_0y'}+(q^{-L(s_0)}-q^{L(s_0)})R^\pi_{w', s_0y'}\\
&=R^\pi_{w', y'}
\end{align*}
In the Case (c) $y\in \textbf{E}_{J,s}^{0,-}$. Using Lemma 3.5 and Lemma 5.3, by the induction hypothesis we have
\begin{align*}
R_{y, w}=-q^{L(s)}R_{y, sw}=-q^{L(s_0)}R^\pi_{w_0(sw), w_0y}=-q^{L(s_0)}R^\pi_{s_0w', y'}\\[2pt]
=-q^{L(s_0)}(-q^{-L(s_0)}R^\pi_{w', y'})=R^\pi_{w', y'}.
\end{align*}
Case (d) $y\in \textbf{E}_{J,s}^{0,+}$. Using Lemma 3.5 and 5.3, by the induction hypothesis we have
\[
R_{y, w}=q^{-L(s)}R_{y, sw}=q^{-L(s_0)}R^\pi_{s_0w', y'}=R^\pi_{w', y'}.
\]
(iii) follows (ii).
\end{proof}
\begin{proposition}
For any $y, w\in \textbf{E}_J$ we have $Q_{y, w}=\epsilon_y\epsilon_w \widetilde{P}^\pi_{w',y'}$.
(Analogously $\widetilde{Q}_{y, w}=\epsilon_y\epsilon_w P^\pi_{w',y'}$).
\end{proposition}
\begin{proof}
We argue by induction on $\ell(w)-\ell(y)\geqslant 0$. If $\ell(w)-\ell(y)= 0$ we have $y=w$ and both sides are $1$.
Assume that $\ell(w)-\ell(y)> 0$. Subtracting the identity in ...from that in ...and using induction hypothesis, we obtain
\[
\overline{\epsilon_y\epsilon_wQ_{y, w}}-\overline{\widetilde{P}^\pi_{w',y'}}=\epsilon_y\epsilon_wQ_{y, w}-\widetilde{P}^\pi_{w',y'}
\]
The right hand side is in $A_{>0}$; since it is fixed by the involution bar, it is $0$.
\end{proof}
More precisely, we have the following inversion formulas
\begin{corollary} In the above situation,
\begin{align*}
\sum\limits_{z\in \textbf{E}_J, x\leqslant z\leqslant w}\varepsilon_w\varepsilon_z P_{x, z}\widetilde{P}^\pi_{w',z'}=\delta_{x, w};\\
\sum\limits_{z\in \textbf{E}_J, x\leqslant z\leqslant w}\varepsilon_w\varepsilon_z \widetilde{P}_{x, z}P^\pi_{w',z'}=\delta_{x, w}
\end{align*}
for all $x, w\in \textbf{E}_J$.
\end{corollary}
\begin{corollary}
If $W$ is finite, for any $y, w\in \textbf{E}_J$ we have

\[
m^s_{y,w}=-\epsilon_{w_0y}\epsilon_{w_0w}m^{\pi, s_0}_{w_0w, w_0y},
\]
where $m^s_{y,w}$ are the elements involved in the multiplication formulas for $C$-basis,
$m^{\pi, s_0}_{w_0w, w_0y}$ are the analogous in the formulas for $C^\pi$-basis.
\end{corollary}
\begin{corollary}
If $W$ is finite, for the bases $D$ and $C^\pi$ in $M(D_J, L)$, and the $\widetilde{D}$-basis and $\widetilde{C}^\pi$-basis for $\widetilde{M}(D_J, L)$
 we have
\[
T_{w_0}D_z=\epsilon_{w_0z} \theta(\widetilde{C}^\pi_{w_0z})\text{ and $T_{w_0}\widetilde{D}_z=\epsilon_{w_0z} \eta(C^\pi_{w_0z})$}.
\]
\end{corollary}


\section{Some remarks}

\subsection*{An example: the dual Solomon modules}
In this subsection, let $(W, S)$ be a finite Coxeter group system. Assume that $L(s)>0$ for all $s\in S$. In \cite{Yin} we introduced the $A$-free $\Hecke$-module
$\Hecke C_{w_J}C'_{w_{\hat{J}}}$, which is called the \textbf{Solomon module} with respect to $J$ and $L$, and where
\begin{align*}
C_{w_J}=\epsilon_{w_J}\sum_{w \in W_J}\epsilon_w q^{L(ww_J)}T_w=\epsilon_{w_J}q^{L(w_J)}\sum_{w \in W_J}\epsilon_w q^{-L(w)}T_w;\\
\,\,C'_{w_{\hat{J}}}
=\sum_{w \in W_{\hat{J}}}q^{-L(ww_{\hat{J}})}T_w=q^{-L(w_{\hat{J}})}\sum_{w \in W_{\hat{J}}}q^{L(w)}T_w.
\end{align*}
that is, $C'_{w_{\hat{J}}}$ is the $C'$-basis element corresponding to $w_{\hat{J}}$, the maximal length element of $W_{\hat{J}}$,
or $c$-basis element corresponding to $w_{\hat{J}}$ (see \cite[Corollary 12.2]{Lusztig2}).
$C_{w_J}$ is the $C$-basis element corresponding to $w_J$.

In \cite{Yin} we showed that $\Hecke C_{w_J}C'_{w_{\hat{J}}}$ has basis $\{\,
T_xC_{w_J}C'_{\hat{J}} \mid x \in F_J\,\}$. 
This basis admits the multiplication rules listed in the Definition 2.4, and $F_J$ is a $W$-graph ideal with respect to $J$ and weight function $L$.

Similarly, the $\Hecke$-module $\Hecke C'_{w_J}C_{w_{\hat{J}}}$ has basis $\{\,
T_xC'_{w_J}C_{\hat{J}} \mid x \in F_J\,\}$. We can easily prove that this basis admits the multiplication rules listed in the Definition 3.1.
We call this the \textbf{dual module} of $\Hecke C_{w_J}C'_{w_{\hat{J}}}$.

\subsection*{The Kazhdan-Lusztig construction} Assume that $J=\emptyset$. Then $D_J=W$ and the sets $WD_J(w)$ and $WA_J(w)$ are empty for all $w\in W$.

(1). If $L(s)>0$ (for all $s\in S$), both modules
$M(\textbf{E}_J, L)$ and $\widetilde{M}(\textbf{E}_J, L)$ are with $A$-basis $(X_w \mid w\in \textbf{E}_J )$ such that,
\begin{equation*}
T_s X_w=\begin{cases}
X_{sw}&\text{if $\ell(sw)>\ell(w)$}\\[3 pt]
X_{sw}+(q^{L(s)}-q^{-L(s)} )X_w&\text{if $\ell(sw)<\ell(w)$,}
\end{cases}
\end{equation*}
where the elements $X_w$ stand for $\Gamma_w$ or $\widetilde{\Gamma}_w$.
If we let $X_w=T_w$ for all $w\in W$, then both modules are the regular module $\Hecke$ with weight function $L$.
Thus we can recover some of Lusztig's results ( for example, see~\cite[Ch.5, 6, 10, 11]{Lusztig2})for the regular case.

\subsection*{Deodhar's construction: the parabolic cases}
Let $J$ be an arbitrary subset of $S$ and $L(s)=1$ for all $s\in S$, we can now turning to Deodhar's construction.

Set $\textbf{E}_J:=D_J$, then $D_J$ is a $W$-graph ideal with respect to $J$, and also it is a $W$-graph ideal with respect
with $\emptyset$.

In the latter case we have $D_{\emptyset}=W$, if $w \in \textbf{E}_J$ then
\begin{align*}
SA(w) &= \{ s \in S | \text{$sw > w$ and $sw \in D_J$} \}, \\
SD(w) &= \{ s \in S | sw < w \},\\
WD_{\emptyset}(w) &= \{ s \in S | sw \notin D_{\emptyset} \} = \emptyset,\\
WA_{\emptyset}(w) &= \{ s \in S | sw \in D_{\emptyset} \ DJ \} = \{ s \in S | \text{$sw = wt$ for some $t \in J$} \}.
\end{align*}

Let $\Hecke_J$ be the Hecke algebra associated with the Coxeter system $(W_J, J)$. Let $M_\psi=\Hecke\otimes_{\Hecke_J}A_\psi$,
where $A_\psi$ is $A$ made into an $\Hecke_J$-module via the homomorphism $\psi: \Hecke_J\rightarrow A$
that satisfies $\psi(T_u)=q^{\ell(u)}$ for all $u\in W_J$, it is a $A$-free with basis $B=\{b_w\mid w\in D_J\}$ defined by $b_w=T_w\otimes 1$.
This corresponds to $M^J$ in~\cite{Deodhar} in the case $u=q$(we note that this is denoted by $\widetilde{M}^J$ in ~\cite{Deodhar2}).

Let $M_\phi=\Hecke\otimes_{\Hecke_J}A_\phi$,
where $A_\psi$ is $A$ made into an $\Hecke_J$-module via the homomorphism $\phi: \Hecke_J\rightarrow A$
that satisfies $\psi(T_u)=(-q)^{-\ell(u)}$ for all $u\in W_J$,
again it is a $A$-free with basis $B=\{b_w\mid w\in D_J\}$ defined by $b_w=T_w\otimes 1$. This corresponds to $M^J$ in~\cite{Deodhar} in the case $u=-1$(this is denoted by $M^J$ in ~\cite{Deodhar2}).

Our module $M(\textbf{E}_J, L)$ is now essentially reduced to be the module $M_\psi$,  while $\widetilde{M}(\textbf{E}_J, L)$ is reduced to be the module $M_\phi$, the only differences being due to our non-traditional definition of $\Hecke$.

In the case $D_J$ is a $W$-graph ideal with respect to $J$, the discussion is similar with the above. For more details see~\cite[Sect. 8]{Howlett}.

\end{document}